\documentclass[11pt,reqno]{amsart}
\usepackage{amsfonts}
\usepackage{mathrsfs}
\usepackage{amsmath}
\usepackage{amssymb,amsfonts}

\newtheorem{thm}{Theorem}[section]
\newtheorem{lem}{Lemma}[section]

\newtheorem{prop}{Proposition}[section]
\numberwithin{equation}{section}
\newtheorem{rmk}{Remark}[section]

\newcommand{\mysection}[1]{\section{#1}\setcounter{equation}{0}}

\newfont{\bb}{msbm10 at 11pt}

\newcommand{\bal}{\begin{aligned}}      \newcommand{\eal}{\end{aligned}}
\newcommand{\ba}{\begin{array}}      \newcommand{\ea}{\end{array}}
\newcommand{\bc}{\begin{center}}     \newcommand{\ec}{\end{center}}
\newcommand{\be}{\begin{enumerate}}  \newcommand{\ee}{\end{enumerate}}
\newcommand{\beq}{\begin{eqnarray}}  \newcommand{\eeq}{\end{eqnarray}}
\newcommand{\beQ}{\begin{eqnarray*}} \newcommand{\eeQ}{\end{eqnarray*}}
\newcommand{\bi}{\begin{itemize}}    \newcommand{\ei}{\end{itemize}}
\newcommand{\bt}{\begin{tabular}}    \newcommand{\et}{\end{tabular}}
\newcommand{\bdm}{\begin{displaymath}} \newcommand{\edm}{\end{displaymath}}

\newcommand{\ls}{\setlength{\baselineskip}{12pt}
                 \setlength{\parskip}{3mm}}

\begin{document}

\title[Horowitz-Myers Conjecture]{The positive energy conjecture for a class of AHM metrics on $\mathbb{R}^{2}\times\mathbb{T}^{n-2}$}

\author[Z Liang]{Zhuobin Liang$^{\dag}$}
\address[]{$^{\dag}$Department of Mathematics, Jinan University, Guangzhou, 510632, P. R. China}
\email{tzbliang@jnu.edu.cn}
\author[X Zhang]{Xiao Zhang$^{\flat}$}
\address[]{$^{\flat}$Guangxi Center for Mathematical Research, Guangxi University, Nanning, Guangxi 530004, PR China}
\address[]{$^{\flat}$Institute of Mathematics, Academy of Mathematics and Systems Science, Chinese Academy of Sciences,
Beijing 100190, PR China and School of Mathematical Sciences, University of Chinese Academy of Sciences, Beijing 100049, PR China}
\email{xzhang@gxu.edu.cn, xzhang@amss.ac.cn}

\date{}

\begin{abstract}
We prove the positive energy conjecture for a class of asymptotically Horowitz-Myers (AHM) metrics on $\mathbb{R}^{2}\times\mathbb{T}^{n-2}$.
This generalizes the previous results of Barzegar-Chru\'{s}ciel-H\"{o}rzinger-Maliborski-Nguyen \cite{BCHMN} as well as the authors \cite{LZ}.
\end{abstract}

\maketitle \pagenumbering{arabic}

\mysection{Introduction}
\ls

In \cite{HM}, Horowitz-Myers constructed so-called AdS solitons with toroidal topology
\beQ
-r^{2}dt^{2}+\frac{1}{\frac{r^{2}}{\ell^{2}}\left(1-\frac{\breve{r} _{0}^{n}}{r^{n}}\right)}dr^{2}+\frac{r^{2}}{\ell^{2}}\left(1-\frac{\breve{r} _{0}^{n}}{r^{n}}\right)d\xi^{2}
+r^{2}\sum_{i=3} ^{n}(d\phi ^{i})^{2}  \label{eq:ads_soliton_spacetime_metric}
\eeQ
which are globally static vacuum $(n+1)$-dimensional spacetime with cosmological constant
\beQ
\Lambda=-\frac{n(n-1)}{2\ell^{2}}.
\eeQ
For convenience, we assume $\ell =1$ throughout the paper. These metrics are geodesically complete when
\beQ
r \in [\breve{r} _{0},\infty), \quad \xi \in [0, \beta_0], \quad \phi^{i}\in[0,\lambda_{i}]
\eeQ
for constants $\breve{r} _0 >0$, $\lambda_{i}>0$, $i=3,\dots,n$ and
\beq
\beta _0=\frac{4\pi }{n \breve{r}_{0}}. \label{beta0}
\eeq
The induced Riemannian metrics $g_{\text{HM}}$ on the constant time slices
\beq
g_{\text{HM}}=\frac{1}{r^{2}\left(1-\frac{\breve{r}_{0}^{n}}{r^{n}}\right)}dr^{2}+r^{2}\left(1-\frac{\breve{r}_{0}^{n}}{r^{n}}\right)d\xi^{2}+r^{2}\sum_{i=3} ^{n}(d\phi ^{i})^{2}\label{eq:ads_soliton_t-slice_g}
\eeq
are asymptotically locally hyperbolic (ALH) and are referred as Horowitz-Myers metrics. Horowitz-Myers also verified that the Hawking-Horowitz mass of $g_{\text{HM}}$ is negative and conjectured that, among all metrics which are asymptotic to $g_{\text{HM}}$ with the same period $\beta _0$ for $\xi$ and with the scalar curvature
\beq
R \geq -n(n-1), \label{S}
\eeq
$g_{\text{HM}}$ is the unique metric with the least Hawking-Horowitz mass \cite{HM}.

Throughout the paper, we misuse the notation and refer a function $f$ type $O(\frac{1}{r^{m}})$ if for \textit{all}
$k\geq0$, its $k$th partial derivative with respect to $r$ satisfies
\begin{equation}
\partial_{r}^{k}f\sim\frac{1}{r^{m+k}},\qquad r\to\infty.\label{eq:decay}
\end{equation}
Furthermore, a tensor is referred as type $O(\frac{1}{r^{m}})$ if its components belong to $O(\frac{1}{r^{m}})$ in coordinates $\{r,\xi,\phi^{i}\}$.

As both Schoen-Yau's method and Witten's method are hard to apply in this situation, there is less progress towards proof of the conjecture. In 2020, Barzegar-Chru\'{s}ciel-H\"{o}rzinger-Maliborski-Nguyen studied the following asymptotically Horowitz-Myers metrics on $\mathbb{R}^{2}\times\mathbb{T}^{n-2}$
\beq
g_1= e^{2u(r)} dr^2 + e^{2v(r)} d\xi^{2}+ e^{2w(r)}\sum_{i=3}^{n}(d\phi ^{i})^{2}, \label{g1}
\eeq
where $u$, $v$, $w$ are functions of $r$ which satisfy
\beq
\begin{aligned}
u(r) & = -\ln r -\frac{1}{2}\ln \left(1-\frac{\breve{r}_0 ^n}{r^n}\right)+ \frac{u_n}{r^n}+O\big(\frac{1}{r^{n+1}}\big),\\
v(r) & = \ln r +\frac{1}{2}\ln \left(1-\frac{\breve{r}_0 ^n}{r^n}\right)+ \frac{v_n}{r^n}+O\big(\frac{1}{r^{n+1}}\big),\\
w(r) & = \ln r+\frac{w_n}{r^n}+O\big(\frac{1}{r^{n+1}}\big),  \label{ag1}
\end{aligned}
\eeq
for constants $u_n$, $v_n$ and $w_n$. The metric $g_1$ is asymptotic to $g_{\text{HM}}$ up to order $O\big(\frac{1}{r^n}\big)$.
In \cite{BCHMN}, they verified Horowitz-Myers conjecture for $g_1$ if its scalar curvature satisfies (\ref{S}) with $\ell =1$ .

Let $a$ be certain constant and $r_{+}$ be the largest positive root of the equation
\beQ
1+\frac{a}{r^{n-1}}-\frac{r_{0}^{n}}{r^{n}}=0
\eeQ
where constant $r_0>0$. Denote
\beq
\beta=\frac{4\pi}{r_{+}\left(n-1+\frac{r_{0}^{n}}{r_{+}^{n}}\right)}.\label{eq:period_xi}
\eeq
In \cite{LZ}, we construct metrics of Horowitz-Myers type with the negative constant scalar curvature
\beq
\mathring{g}=\frac{1}{r^{2}\bigg(1+\frac{a}{r^{n-1}}-\frac{r_{0}^{n}}{r^{n}}\bigg)}dr^{2}+r^{2}\bigg(1+\frac{a}{r^{n-1}}-\frac{r_{0}^{n}}{r^{n}}\bigg)d\xi^{2}+r^{2}\sum_{i=3}^{n}(d\phi^{i})^{2}, \label{eq:HM_type_metric}
\eeq
where
\beQ
r\in[r_{+},\infty),\quad  \xi\in[0,\beta], \quad  \phi^{i}\in[0,\lambda_{i}].
\eeQ
This is a geodesically complete metric on $\mathbb{R}^{2}\times\mathbb{T}^{n-2}$ of constant scalar curvature $-n(n-1)$. And it is asymptotic to $g_{\text{HM}}$ up to order $O\big(\frac{1}{r^{n-1}}\big)$ provided $\beta_{0}=\beta$. As $\mathring{g}$ decays weaker than $g_1$, we can not apply their positive energy theorem to $\mathring{g}$ directly. However, we can still verify that the Horowitz-Myers conjecture holds for $\mathring{g}$ \cite{LZ}.

Now we study the following more general perturbation of Horowitz-Myers metrics on $\mathbb{R}^{2}\times\mathbb{T}^{n-2}$
\beq
g= e^{2u(r)} dr^2 + e^{2v(r,\xi)} d\xi^{2}+ r^2 \left( \sum_{i=3} ^{n} (d\phi^{i})^2 +\hat w (r, \xi, \phi ^i) \right), \label{eq:g}
\eeq
where
\beq
\begin{aligned}
u(r) & = -\ln r -\frac{1}{2}\ln \left(1+\frac{a}{r^{n-1}}-\frac{r_0 ^n}{r^n}\right)+\hat u(r),\\
v(r, \xi) &= \ln r +\frac{1}{2}\ln \left(1+\frac{a}{r^{n-1}}-\frac{r_0 ^n}{r^n}\right)+\hat v(r, \xi),\\
e^{\hat u (r)} &= 1+\frac{u_{n-1}}{r^{n-1}} +\frac{u_n}{r^n}+O\big(\frac{1}{r^{n+1}}\big),\\
e^{\hat v (r, \xi)} &= 1+\frac{v_{n-1} (\xi) }{r^{n-1}} +\frac{v_n(\xi) }{r^n}+O\big(\frac{1}{r^{n+1}}\big),\\
\hat w (r, \xi, \phi ^i) &= \frac{2 w_{n-1} (\xi, \phi ^i) }{r^{n-1}} +\frac{2 w_n (\xi, \phi ^i) }{r^n}+O\big(\frac{1}{r^{n+1}}\big),\label{eq:ag}
\end{aligned}
\eeq
and $u_{n-1}$, $u_{n}$ are real constants, $v_{n-1}(\xi)$, $v_{n}(\xi)$ are smooth functions of $\xi$, $w_{n-1}(\xi, \phi^i)$ and $w_{n}(\xi, \phi^i)$ are 2-tensors on $\mathbb{T}^{n-2}$
which depend on $\xi$, $\phi^{i}$. The restriction of $g$ on $\mathbb{T}^{n-2}$ is denoted as
\beQ
\gamma=r^2 \left( \sum_{i=3} ^{n} (d\phi^{i})^2 +\hat w (r, \xi, \phi ^i) \right).
\eeQ

In this paper, we show that the Horowitz-Myers conjecture is true for the above AHM metrics.

\begin{thm}
\label{thm:Main} Let $n\geq3$ and let $g$ be the smooth metric (\ref{eq:g}) on $\mathbb{R}^{2}\times\mathbb{T}^{n-2}$, which satisfies (\ref{eq:ag}).
Suppose the scalar curvature $R(g)$ of $g$ satisfies
\begin{equation}
R(g)+n(n-1)\geq0, \quad \big[R(g)+n(n-1)\big]r\in L^{1}(\mathbb{R}^{2}\times\mathbb{T}^{n-2}, g),  \label{eq:Main_thm_scalar_curv_cond}
\end{equation}
and suppose
\beq
\beta_{0}=\beta.
\label{eq:beta0=beta}
\eeq
If $n\geq4$, we assume in addition that the scalar curvature $R(\gamma)$ of $\gamma$ satisfies
\begin{equation}
\int_{\mathbb{T}^{n-2}}R(\gamma)d\phi^{3}\cdots d\phi^{n}\leq0  \label{eq:Main_thm_scalar_cur_torus}
\end{equation}
for all $r$ and $\xi$.
Then the total energy
\beQ
E(g) - E(g_{\text{HM}})\geq 0.
\eeQ
The equality holds if and only if $g$ is isometric to $g_{\text{HM}}$.
\end{thm}

\begin{rmk}
The assumption that (\ref{eq:decay}) holds for \textit{all} $k\geq0$ is only for the sake of simplicity of notations appeared in the paper.
Indeed, the proof of Theorem \ref{thm:Main} needs only $k =0,1,2$.
\end{rmk}

\begin{rmk}
Condition (\ref{eq:Main_thm_scalar_cur_torus}) is trivial when $n=3$, since in this case, $\gamma$
is a one dimensional metric with $R(\gamma)\equiv0$. If $n=4$, then the Gauss-Bonnet theorem states that
\beQ
\int_{\mathbb{T}^{2}}R(\gamma)dV_{\gamma}=0,
\eeQ
which can not implies (\ref{eq:Main_thm_scalar_cur_torus}).
\end{rmk}

\begin{rmk}
In \cite{BCHMN}, the components of $\gamma$ are assumed to be independent of $\phi^{i}$ (e.g. (\ref{g1})). In
this case, $\gamma$ is a flat metric on $\mathbb{T}^{n-2}$ and (\ref{eq:Main_thm_scalar_cur_torus}) holds for all $n$.
\end{rmk}

The paper is organized as follows. In Section 2, we show that $g$ can be written as an asymptotically Poincar\'{e}-Einstein (APE) metric. Then the total energy can be computed in terms of the definitions of Wang \cite{Wa} and
Chru\'{s}ciel-Herzlich \cite{CH}. In Section 3, we prove the Horowitz-Myers conjecture for $g$.

\mysection{APE metrics}
\ls

In this section, we show that $g$ is actually an APE metric. The boundary-at-infinity $\partial_{\infty}M$ of $\mathbb{R}^{2}\times\mathbb{T}^{n-2}$ is $\mathbb{T}^{n-1}$.
Let $x$ be a special defining function for $\partial_{\infty}M$. That is, on a collar neighborhood of $\partial_{\infty}M$,

(1) $x\geq0$ with $\{x=0\}=\partial_{\infty}M$,

(2) $\vert dx\vert_{x^{2}g}\equiv1$.\\\\
On this collar neighborhood, the metric $g$ can be written in the form
\[
g=\frac{1}{x^{2}}(dx^{2}\oplus h_{x}),\quad \lim _{r \to\infty} xr=1
\]
where $h_{x}$ are the induced metrics of $x^{2}g$ on the constant $x$ slices,
\[
\frac{1}{x}dx=-\frac{e^{\hat{u}(r)}}{r\sqrt{1+\frac{a}{r^{n-1}}-\frac{r_{0}^{n}}{r^{n}}}}dr.
\]
It follows that
\beQ
\begin{aligned}
\ln x ={} &  -\int_{r_{+}}^{r}\frac{e^{\hat{u}(s)}}{s\sqrt{1+\frac{a}{s^{n-1}}-\frac{r_{0}^{n}}{s^{n}}}}ds+C\\
      ={} & -\ln r-\int_{r_{+}}^{r} \left( \frac{u_{n-1}-\frac{1}{2}a}{s^{n}}+\frac{u_{n}+\frac{1}{2}r_{0}^{n}}{s^{n+1}}+O(s^{-n-2})  \right)ds\\
        & +\ln r_{+}+C,
\end{aligned} \label{eq:defining_func_x}
\eeQ
where
\beQ
C =-\ln r_{+}+\int_{r_{+}}^{\infty}\frac{1}{s}\left(\frac{e^{\hat{u}(s)}}{\sqrt{1+\frac{a}{s^{n-1}}-\frac{r_{0}^{n}}{s^{n}}}}-1\right)ds.
\eeQ
Therefore, we obtain
\[
r=x^{-1}+\frac{u_{n-1}-\frac{1}{2}a}{n-1}x^{n-2}+\frac{u_{n}+\frac{1}{2}r_{0}^{n}}{n}x^{n-1}+O(x^{n}).
\]
Putting this into (\ref{eq:g}) and applying the asymptotic conditions (\ref{eq:ag}), we obtain
\begin{align}
g ={} &  \frac{1}{x^{2}}\bigg\{ dx^{2}+d\xi^{2}+\sum_{i=3}^{n}(d\phi^{i})^{2}   \nonumber  \\
  & +\frac{x^{n-1}}{n-1}\bigg[\big((n-2)a+2u_{n-1}+2(n-1)v_{n-1}\big)d\xi^{2}  \nonumber  \\
  & +\big(2u_{n-1}-a\big)\sum_{i=3}^{n}(d\phi^{i})^{2}+2(n-1)w_{n-1}\bigg]   \label{eq:g_normal_expansion} \\
  & +\frac{x^{n}}{n}\bigg[(-(n-1)r_{0}^{n}+2u_{n}+2nv_{n})d\xi^{2}   \nonumber \\
  & +(r_{0}^{n}+2u_{n})\sum_{i=3}^{n}(d\phi^{i})^{2}+2nw_{n}\bigg]+O(x^{n+1})\bigg\}. \nonumber
\end{align}
Denote
\begin{align*}
h_{0}   ={}& d\xi^{2}+\sum_{i=3}^{n}(d\phi^{i})^{2},\\
\theta   ={}& \Big((n-2)a+2u_{n-1}+2(n-1)v_{n-1}\Big)d\xi^{2}\\
          &+(2u_{n-1}-a)\sum_{i=3}^{n}(d\phi^{i})^{2}+2(n-1)w_{n-1},\\
\kappa  ={}& \Big(-(n-1)r_{0}^{n}+2u_{n}+2nv_{n}\Big)d\xi^{2}\\
         & +\big(r_{0}^{n} +2u_{n}\big)\sum_{i=3}^{n}(d\phi^{i})^{2}+2nw_{n}.
\end{align*}

The conformal infinity of $g$ is defined as the conformal class $[h_{0}]$ on $\partial_{\infty}M$ (cf. \cite{GSW}). The mass aspect function, the total energy are defined as
\begin{align}
\text{tr}_{h_{0}}\kappa &=-r_{0}^{n}+2(n-1)u_{n}+2nv_{n}+2n\text{tr}_{h_{0}}w_{n}, \label{eq:mass_aspect_function} \\
E(g)&=\int_{\partial_{\infty}M}\text{tr}_{h_{0}}\kappa dV_{h_{0}}\label{eq:mass}
\end{align}
respectively \cite{Wa, CH}. When $u_{n}=v_{n}=w_{n}=0$, we obtain
\beQ
E(g_{\text{HM}})=-\breve{r}_{0}^{n}\text{Vol}(\partial_{\infty}M,h_{0}).
\eeQ
Therefore
\begin{equation}
E(g)-E(g_{\text{HM}})=\int_{\partial_{\infty}M}\bigg[\breve{r}_{0}^{n}-r_{0}^{n}+2(n-1)u_{n}+2nv_{n}+2n\text{tr}_{h_{0}}w_{n}\bigg]dV_{h_{0}}.\label{eq:mass_diff}
\end{equation}

Now we express the scalar curvature $R(g)$ in some appropriate form. Denote
\[
W^{r}=\frac{1}{2}\text{tr}_{\gamma}(\partial_{r}\gamma)=\frac{1}{2}(\partial_{r}\gamma_{ij})\gamma^{ij}
\]
and let $\mathring{W} ^{r}$ be the trace-less part of $\partial_{r}\gamma$. It is straightforward that
\begin{equation}
\partial_{r}\gamma=\frac{2W^{r}}{n-2}\gamma+\mathring{W} ^{r}.\label{eq:gamma_r_derivatives}
\end{equation}
By the asymptotic condition of $\gamma$, we obtain
\begin{equation}
W^{r}=\frac{n-2}{r}-\frac{(n-1)\text{tr}_{h_{0}}w_{n-1}}{r^{n}}-\frac{n\text{tr}_{h_{0}}w_{n}}{r^{n+1}}+O(r^{-n-2}).\label{eq:Wr_expansion}
\end{equation}
Similarly, denote
\[
W^{\xi}=\frac{1}{2}\text{tr}_{\gamma}(\partial_{\xi}\gamma)=\frac{1}{2}(\partial_{\xi}\gamma_{ij})\gamma^{ij}
\]
and let $\mathring{W}^{\xi}$ be the trace-less part of $\partial_{\xi}\gamma$. We can also get
\begin{equation}
\partial_{\xi}\gamma=\frac{2W^{\xi}}{n-2}\gamma+\mathring{W}^{\xi}.\label{eq:gamma_xi_derivatives}
\end{equation}

Set
\begin{equation}
\hat{W}^{r}=W^{r}-\frac{n-2}{r}.\label{eq:Wr_hat}
\end{equation}

\begin{lem}
\label{lem:curvature_formula} With the above notations, the scalar curvature
$R(g)$ of $g$ is
\begin{equation}
\begin{aligned}
R(g) ={}& -2e^{-u-v}\partial_{r}\bigg[e^{v-u}(\partial_{r}v+W^{r})\bigg]-2e^{-v}\partial_{\xi}\big(e^{-v}W^{\xi}\big)\\
      &+R(\gamma) -2e^{-2u}\bigg[\frac{n-1}{2(n-2)}(W^{r})^{2}+\frac{1}{8}\vert \mathring{W} ^{r}\vert_{\gamma}^{2}\bigg]\\
      & -2e^{-2v}\bigg[\frac{n-1}{2(n-2)}(W^{\xi})^{2}+\frac{1}{8}\vert \mathring{W}^{\xi}\vert_{\gamma}^{2}\bigg].
\end{aligned}
\label{eq:scalar_curvature}
\end{equation}
\end{lem}

\begin{proof}
For the following metrics
\[
N(s)^{2}ds^{2}+\eta, \quad \eta=\eta_{ab}(s,y^{1},\cdots,y^{m})dy^{a}dy^{b},
\]
it is well-known that their scalar curvature is
\begin{equation}
R(\eta)-N^{-1}\big(N^{-1}\text{tr}_{\eta}\eta^{\prime}\big)^{\prime}-\frac{1}{4N^{2}}\vert\eta^{\prime}\vert_{\eta}^{2}-\frac{1}{4N^{2}}\big(\text{tr}_{\eta}\eta^{\prime}\big)^{2},\label{eq:lem_scalar_curv}
\end{equation}
where $R(\eta)$ is the scalar curvature of $\eta$, $^{\prime}$ is the partial derivative with respect to $s$,
\beQ
\eta^{\prime}=\eta_{ab}^{\prime}dy^{a}dy^{b}.
\eeQ
Now we apply formula (\ref{eq:lem_scalar_curv}) to the following part of $g$
\[
\rho=e^{2v}d\xi^{2}+\gamma,
\]
and obtain
\beQ
\begin{aligned}
R(\rho)  &= R(\gamma)-e^{-v}\partial_{\xi}\big(e^{-v}\text{tr}_{\gamma}\partial_{\xi}\gamma\big)
                          -\frac{1}{4}e^{-2v}\vert\partial_{\xi}\gamma\vert_{\gamma}^{2}-\frac{1}{4}e^{-2v}\big(\text{tr}_{\gamma}\partial_{\xi}\gamma\big)^{2}\\
                        &= R(\gamma)-2e^{-v}\partial_{\xi}\big(e^{-v}W^{\xi}\big)
                          -2e^{-2v}\bigg[\frac{n-1}{2(n-2)}(W^{\xi})^{2}+\frac{1}{8}\vert \mathring{W}^{\xi}\vert_{\gamma}^{2}\bigg].
\end{aligned}
\label{eq:lem_scalar_curv_rho}
\eeQ
Next we apply (\ref{eq:lem_scalar_curv}) to the metric
\[
g=e^{2u}dr^{2}+\rho
\]
and obtain
\[
R(g)=R(\rho)-e^{-u}\partial_{r}\big(e^{-u}\text{tr}_{\rho}\partial_{r}\rho\big)-\frac{1}{4}e^{-2u}\vert\partial_{r}\rho\vert_{\rho}^{2}-\frac{1}{4}e^{-2u}\big(\text{tr}_{\rho}\partial_{r}\rho\big)^{2}.
\]
Since
\begin{align*}
\partial_{r}\rho & =2(\partial_{r}v)e^{2v}d\xi^{2}+\partial_{r}\gamma,\\
\text{tr}_{\rho}\partial_{r}\rho & =2(\partial_{r}v)+\text{tr}_{\gamma}\partial_{r}\gamma=2(\partial_{r}v)+2W^{r},\\
\vert\partial_{r}\rho\vert_{\rho}^{2} & =4(\partial_{r}v)^{2}+\vert\partial_{r}\gamma\vert_{\gamma}^{2}=4(\partial_{r}v)^{2}+\frac{4}{n-2}(W^{r})^{2}+\vert \mathring{W} ^{r}\vert_{\gamma}^{2}.
\end{align*}
we obtain
\begin{align*}
R(g)  ={}& R(\rho)-2e^{-u}\partial_{r}\big[e^{-u}(\partial_{r}v+W^{r})\big]\\
       & -e^{-2u}\bigg[2(\partial_{r}v)^{2}+\frac{n-1}{n-2}(W^{r})^{2}+2\partial_{r}vW^{r}+\frac{1}{4}\vert \mathring{W} ^{r}\vert_{\gamma}^{2}\bigg]\\
      ={}& R(\rho)-2e^{-u-v}\partial_{r}\big[e^{v-u}(\partial_{r}v+W^{r})\big]\\
       & -2e^{-2u}\bigg[\frac{n-1}{2(n-2)}(W^{r})^{2}+\frac{1}{8}\vert \mathring{W} ^{r}\vert_{\gamma}^{2}\bigg].
\end{align*}
Therefore the lemma follows.
\end{proof}

\begin{prop}\label{prop:scalar_curvature_L1}
For the metric (\ref{eq:g}), the condition
\beQ
\big[R(g)+n(n-1)\big]r \in L^{1}(\mathbb{R}^{2}\times\mathbb{T}^{n-2}, g)
\eeQ
holds if and only if
\begin{equation}
\text{tr}_{h_{0}}\theta=0 \Longleftrightarrow u_{n-1}+v_{n-1}+\text{tr}_{h_{0}}w_{n-1}=0.\label{eq:scalar_cur_L1_condition}
\end{equation}
In this case,
\begin{equation}
R+n(n-1)=O\big(\frac{1}{r^{n+1}}\big).\label{eq:scalar_curvature_more_integrability}
\end{equation}
\end{prop}

\begin{proof}
Note if all $\hat{u}$, $\hat{v}$, $\hat{W}^{r}$, $W^{\xi}$,
$\mathring{W} ^{r}$ and $\mathring{W}^{\xi}$ vanish, then $g$ reduces to
\[
e^{-2f}dr^{2}+e^{2f}d\xi^{2}+r^{2}\sum_{i,j=3}^{n}\gamma_{ij}(\phi^{3},\cdots,\phi^{n})d\phi^{i}d\phi^{j},
\]
where
\beQ
f=\ln r+\frac{1}{2}\ln \left(1+\frac{a}{r^{n-1}}-\frac{r_{0}^{n}}{r^{n}}\right).
\eeQ
By (\ref{eq:scalar_curvature}), the scalar curvature of this metric is
\beQ
-n(n-1)+R(\gamma).
\eeQ
Putting (\ref{eq:ag}) and (\ref{eq:Wr_hat}) into (\ref{eq:scalar_curvature}),
we obtain
\begin{align*}
R(g)={} & -n(n-1)e^{-2\hat{u}}+2e^{2f-2\hat{u}}\bigg[-\partial_{r}\big(\partial_{r}\hat{v}+\hat{W}^{r}\big)\\
      &-\partial_{r}\big(\hat{v}-\hat{u}+2f+(n-1)\ln r\big)\cdot\big(\partial_{r}\hat{v}+\hat{W}^{r}\big)\label{eq:scalar_curvature_good_form}\\
      & +\big(\frac{1}{r}-\partial_{r}f\big)\partial_{r}\hat{v}+\big(\partial_{r}f+\frac{n-2}{r}\big)\partial_{r}\hat{u}\\
      & -\frac{(n-1)}{2(n-2)}(\hat{W}^{r})^{2}-\frac{1}{8}\vert \mathring{W} ^{r}\vert_{\gamma}^{2}\bigg] \\
      & -2e^{-v}\partial_{\xi}\big(e^{-v}W^{\xi}\big)+R(\gamma)\\
      & -2e^{-2v}\bigg[\frac{n-1}{2(n-2)}(W^{\xi})^{2}+\frac{1}{8}\vert \mathring{W}^{\xi}\vert_{\gamma}^{2}\bigg]\\
    ={} &-n(n-1)+(\text{tr}_{h_{0}}\theta)r^{1-n}+O(r^{-n-1}).
\end{align*}
Here we use that
\beQ
R(\gamma)=O(r^{-n-1}),
\eeQ
which could be confirmed by direct computation. Since
\beQ
\sqrt{\det g}=O(r^{n-2}),
\eeQ
and
\beQ
\text{tr}_{h_{0}}\theta=2(n-1)(v_{n-1}+u_{n-1}+\text{tr}_{h_{0}}w_{n-1}),
\eeQ
so the lemma follows.
\end{proof}

By (\ref{eq:g_normal_expansion}), we have
\[
g=x^{-2} \bar g, \quad \bar g=dx^{2}+d\xi^{2}+\sum_{i=3}^{n}(d\phi^{i})^{2}+\frac{x^{n-1}}{n-1}\theta+O(x^{n}).
\]
We denote $\phi^{0}=x$ and $\phi^{1}=\xi$. Let $a$, $b$ and $c$ be indices ranging from $1$ to $n$. Under the coordinates $\{\phi^{0},\phi^{1},\cdots,\phi^{n}\}$, it holds
\beQ
\bar{g} _{00}=1, \quad \bar{g} _{0a}=0, \quad \bar{g} _{ab}=\delta_{ab}+\frac{x^{n-1}}{n-1}\theta_{ab}+O(x^{n}).
\eeQ
It is direct to show that the Christoffel symbols of $\bar g$ satisfy
\begin{align*}
\Gamma_{ab}^{c} &= O(x^{n-1}),\\
\Gamma_{ab}^{0} &= -\frac{1}{2}x^{n-2}\theta_{ab}+O(x^{n-1}),\\
\Gamma_{a0}^{b} &= \Gamma_{0a}^{b}=\frac{1}{2}x^{n-2}\theta_{ab}+O(x^{2n-3}),
\end{align*}
and the others vanish. So the Ricci curvature of $\bar{g}$ satisfies
\begin{align*}
\bar{R}_{00} & = -\frac{1}{2}(n-2)x^{n-3}\text{tr}_{h_{0}}\theta+O(x^{2n-4}),\\
\bar{R}_{0a} & = O(x^{n-2}),\\
\bar{R}_{ab} & = O(x^{n-2}).
\end{align*}
Denote by $Ric$ the Ricci curvature of $g$. Let
\[
\Omega=Ric+(n-1)g.
\]
By the formulas of conformal transformation of the Ricci and scalar curvatures, we finally obtain
\begin{align*}
\Omega(\partial_{x},\partial_{x}) & =-\frac{1}{2}(n-3)x^{n-3}\text{tr}_{h_{0}}\theta+O(x^{n-2}),\\
\Omega(\partial_{a},\partial_{b}) & =\frac{1}{2}x^{n-3}(\text{tr}_{h_{0}}\theta)\delta_{ab}+O(x^{n-2}),\\
\Omega(\partial_{x},\partial_{a}) & =O(x^{n-2}).
\end{align*}

Recall Definition 2.1 in \cite{Wo} that the metric $g$ is asymptotically Poincar\'{e}--Einstein
(APE) if
\beQ
\lim _{x \to 0} |\Omega|_{g}=O(x^{n}).
\eeQ
By Lemma 2.2 \cite{Wo}, $g$ is APE if and only if $\text{tr}_{h_{0}}\theta=0$. (This can also be seen from the above computation.)
Therefore we have

\begin{prop}
For the metric (\ref{eq:g}), the condition
\beQ
\big[R(g)+n(n-1)\big]r \in L^{1}(\mathbb{R}^{2}\times\mathbb{T}^{n-2}, g)
\eeQ
holds if and only if $g$ is APE.
\end{prop}

Finally, we study the consequence of smoothness of $g$ at $r=r_{+}$.
\begin{prop}
If the metric $g$ is smooth at $r=r_{+}$, then
\begin{equation}
\hat{v}(r_{+},\xi)=\hat{u}(r_{+})\quad\text{for all }\xi.\label{eq:regularity_cond}
\end{equation}
\end{prop}

\begin{proof}
We follow the proof of Theorem II.1 in \cite{LZ} to construct a regular coordinate system around $r=r_{+}$. Set
\[
A(r)=r^{2}\bigg(1+\frac{a}{r^{n-1}}-\frac{r_{0}^{n}}{r^{n}}\bigg)
\]
for $r>0$. Differentiating $A$ at $r=r_{+}$ and using (\ref{eq:period_xi}), we find
\[
A^{\prime}(r_{+})=\frac{4\pi}{\beta}>0.
\]
This implies that $A$ has a smooth inverse, which we denote by $B$,  round $r_{+}$. Clearly $B(0)=r_{+}$. Now we can introduce new coordinates
$(\rho,\Xi)$ given by
\beQ
\rho=A^{\frac{1}{2}}(r), \quad \Xi=\frac{2\pi\xi}{\beta}
\eeQ
for $r_{+}<r<r_{+}+\varepsilon$, $0 \leq \Xi <2\pi$, with some $\varepsilon>0$.
So $(\rho,\Xi)$ can be viewed as the standard polar coordinates around the origin in $\mathbb{R}^{2}$.
Set
\beQ
z^{1}=\rho\cos\Xi, \quad z^{2}=\rho\sin\Xi.
\eeQ
Obviously, $\{r=r_{+}\}=\{z^{1}=z^{2}=0\}$.
It is shown in \cite{LZ} that $(z^{1},z^{2},\phi^{3},\cdots,\phi^{n})$
is a regular coordinate system around $r=r_{+}$ for $\mathring{g}$.
Since $g$ is smooth, the components in this coordinate system are
smooth around $(z^{1},z^{2})=(0,0)$. For our purpose, we only rewrite
the following part
\[
g_{(0)}:=e^{2\hat{u}(r)}A^{-1}(r)dr^{2}+e^{2\hat{v}(r,\xi)}A(r)d\xi^{2}
\]
of $g$ in the new coordinate system. Since
\[
A^{-1}(r)dr^{2}=\frac{4}{A^{\prime}(r)^{2}}d\rho^{2},\quad  A(r)d\xi^{2}=\frac{4\rho^{2}}{A^{\prime}(r_{+})^{2}}d\Xi^{2},
\]
we have
\[
g_{(0)}=\frac{4e^{2\hat{u}(r)}}{A^{\prime}(r)^{2}}d\rho^{2}+\frac{4e^{2\hat{v}(r,\xi)}}{A^{\prime}(r_{+})^{2}}\rho^{2}d\Xi^{2}=g_{(\rho)}+g_{(\Xi)}+g_{\mathbb{R}^{2}}
\]
where
\begin{align*}
g_{(\rho)}  &= \bigg[\frac{4e^{2\hat{u}(r)}}{A^{\prime}(r)^{2}}-\frac{4e^{2\hat{u}(r_{+})}}{A^{\prime}(r_{+})^{2}}\bigg]d\rho^{2},\\
g_{(\Xi)} &= \frac{4e^{2\hat{u}(r_{+})}}{A^{\prime}(r_{+})^{2}}\bigg[e^{2(\hat{v}(r,\xi)-\hat{u}(r_{+}))}-1\bigg]\rho^{2}d\Xi^{2},\\
g_{\mathbb{R}^{2}} &=\frac{4e^{2\hat{u}(r_{+})}}{A^{\prime}(r_{+})^{2}}(d\rho^{2}+\rho^{2}d\Xi^{2}).
\end{align*}
Since $\frac{4e^{2\hat{u}(r)}}{A^{\prime}(r)^{2}}$ is smooth near
$r=r_{+}$, there exists a smooth function $\varphi(r)$ so that
\[
\frac{4e^{2\hat{u}(r)}}{A^{\prime}(r)^{2}}-\frac{4e^{2\hat{u}(r_{+})}}{A^{\prime}(r_{+})^{2}}=(r-r_{+})\varphi(r)
\]
for $r$ near $r_{+}$. Note that
\[
(r-r_{+})\varphi(r)=\big[B(\rho^{2})-B(0)\big](\varphi\circ B)(\rho^{2}).
\]
Since $B$ is smooth near $0$, there also exists a smooth function
$\psi(\rho)$ such that
\beQ
B(\rho^{2})-B(0)=\rho^{2}\psi(\rho^{2}).
\eeQ
So
\begin{align*}
g_{(\rho)} & =(\varphi\circ B)(\rho^{2})\psi(\rho^{2})\rho^{2}d\rho^{2}\\
 & =((\varphi\circ B)\cdot\psi)((z^{1})^{2}+(z^{2})^{2})\big(z^{1}dz^{1}+z^{2}dz^{2}\big)^{2}.
\end{align*}
This implies that $g_{(\rho)}$ is smooth around $(z^{1},z^{2})=(0,0)$.
As a result, $g_{(\Xi)}$ is also smooth around $(z^{1},z^{2})=(0,0)$.
However,
\begin{align*}
\rho^{2}d\Xi^{2} & =(dz^{1})^{2}+(dz^{2})^{2}-d\rho^{2}\\
 & =(dz^{1})^{2}+(dz^{2})^{2}-\frac{1}{\rho^{2}}\big(z^{1}dz^{1}+z^{2}dz^{2}\big)^{2}
\end{align*}
is singular at $(z^{1},z^{2})=(0,0)$. So it must hold that
\beQ
\lim_{r\rightarrow r_{+}}\hat{v}(r,\xi)=\hat{u}(r_{+})
\eeQ
for all $\xi$. This proves (\ref{eq:regularity_cond}).
\end{proof}

\section{Proof of the theorem}

In this section, we adopt the idea in \cite{BCHMN} to prove Theorem \ref{thm:Main}. We will replace $r$ by a new coordinate function $\tilde{r}$
to change $e^{2u(r)}dr^{2}$ in (\ref{eq:g}) to $\tilde{r}^{-2}\bigg(1-\frac{\tilde{r}_{0}^{n}}{\tilde{r}^{n}}\bigg)^{-1}d\tilde{r}^{2}$.
\begin{lem}
There exist a positive number $\tilde{r}_{0}$ and a smooth increasing
function $r\mapsto\tilde{r}(r)$ from $[r_{+},\infty)$ to $[\tilde{r}_{0},\infty)$,
with $\tilde{r}(r_{+})=\tilde{r}_{0}$, such that
\begin{equation}
\frac{1}{\tilde{r}\sqrt{1-\frac{\tilde{r}_{0}^{n}}{\tilde{r}^{n}}}}d\tilde{r}=\frac{1}{r\sqrt{1+\frac{a}{r^{n-1}}-\frac{r_{0}^{n}}{r^{n}}}}e^{\hat{u}(r)}dr,\label{eq:change_coordinates_grr}
\end{equation}
and
\begin{equation}
\tilde{r}(r)=r+\frac{a-2u_{n-1}}{2(n-1)}r^{2-n}+\frac{\tilde{r}_{0}^{n}-r_{0}^{n}-2u_{n}}{2n}r^{1-n}+O(r^{-n}).\label{eq:change_coordinates_new_r}
\end{equation}
\end{lem}

\begin{proof}
Define
\[
F(r)=\int_{1}^{r}\frac{1}{s\sqrt{1-\frac{1}{s^{n}}}}ds, \quad r\geq 1.
\]
For large $r$, we have
\begin{equation}
F(r)=F_{0}+\ln r-\frac{1}{2n}r^{-n}+O(r^{-n-1}),\label{eq:change_coordinates_F}
\end{equation}
where
\[
F_0=\int_{1}^{\infty}\frac{1}{s}\left(\frac{1}{\sqrt{1-\frac{1}{s^{n}}}}-1\right)ds.
\]
Clearly, $F$ is a smooth increasing function. Define $\tilde{r}(r)$ on $[r_{+},\infty)$ by
\begin{equation}
F\bigg(\frac{\tilde{r}(r)}{\tilde{r}_{0}}\bigg)=\int_{r_{+}}^{r}\frac{e^{\hat{u}(s)}}{s\sqrt{1+\frac{a}{s^{n-1}}-\frac{r_{0}^{n}}{s^{n}}}}ds.\label{eq:change_coordinates_new_r_1}
\end{equation}
Here $\tilde{r}_{0}$ is a positive number to be specified later.
The integral on the right hand side has the following asymptotic expansion
as $r\to\infty$:
\begin{equation}
C+\ln r+\frac{a-2u_{n-1}}{2(n-1)}r^{1-n}-\frac{2u_{n}+r_{0}^{n}}{2n}r^{-n}+O(r^{-n-1}),\label{eq:change_coordinates_expansion}
\end{equation}
where
\begin{align*}
C=&\lim_{r\to\infty}\bigg(\int_{r_{+}}^{r}\frac{1}{s\sqrt{1+\frac{a}{s^{n-1}}-\frac{r_{0}^{n}}{s^{n}}}}ds-\ln r\bigg)\\
 =&-\ln r_{+}+\int_{r_{+}}^{\infty}\frac{1}{s}\left(\frac{1}{\sqrt{1+\frac{a}{s^{n-1}}-\frac{r_{0}^{n}}{s^{n}}}}-1\right)ds.
\end{align*}
Putting (\ref{eq:change_coordinates_F}) and (\ref{eq:change_coordinates_expansion})
into (\ref{eq:change_coordinates_new_r_1}), we obtain
\begin{align*}
F_{0}-&\ln\tilde{r}_{0}+\ln\tilde{r}-\frac{\tilde{r}_{0}^{n}}{2n}\tilde{r}^{-n}+O(\tilde{r}^{-n-1})\\
     =&\; C+\ln r+\frac{a-2u_{n-1}}{2(n-1)}r^{1-n}-\frac{2u_{n}+r_{0}^{n}}{2n}r^{-n}+O(r^{-n-1}).
\end{align*}
Now we set
\[
\tilde{r}_{0}=e^{F_{0}-C}.
\]
The above equation implies the asymptotic expansion (\ref{eq:change_coordinates_new_r}).
The equation (\ref{eq:change_coordinates_grr}) follows directly by
the definition of $\tilde{r}$.
\end{proof}

With the function $\tilde{r}$, we rewrite the metric
\begin{equation}
g=e^{2\tilde{u}(\tilde{r})}d\tilde{r}^{2}+e^{2\tilde{v}(\tilde{r},\xi)}d\xi^{2}+\tilde{r}^{2}\bigg(\sum_{i=3}^{n}(d\phi^{i})^2+\hat{\tilde{w}}(\tilde{r}, \xi, \phi^i)\bigg),\label{eq:g_rewrite_new_coordinates}
\end{equation}
where
\beq
\begin{aligned}
\tilde{u}(\tilde{r}) & =-\ln\tilde{r}-\frac{1}{2}\ln\bigg(1-\frac{\tilde{r}_{0}^{n}}{\tilde{r}^{n}}\bigg),\\
\tilde{v}(\tilde{r},\xi) & =-\tilde{u}(\tilde{r})+\hat{\tilde{v}}(\tilde{r},\xi),\\
e^{\hat{\tilde{v}}(\tilde{r},\xi)}& =1+\frac{\tilde{v}_{n-1}(\xi)}{\tilde{r}^{n-1}}+\frac{\tilde{v}_{n}(\xi)}{\tilde{r}^{n}}+O\big(\frac{1}{\tilde{r}^{n+1}}\big),\\
\hat{\tilde{w}}(\tilde{r}, \xi, \phi^i) & =\frac{2\tilde{w}_{n-1}}{\tilde{r}^{n-1}}+\frac{2\tilde{w}_{n}}{\tilde{r}^{n}}+O\big(\frac{1}{\tilde{r}^{n+1}}\big). \label{eq:v_new_coordinates}
\end{aligned}
\eeq
By comparing (\ref{eq:g}) and (\ref{eq:g_rewrite_new_coordinates}),
we have
\begin{equation}
\hat{\tilde{v}}=\hat{v}+\hat{u}-\ln\frac{d\tilde{r}}{dr}\label{eq:v_title}
\end{equation}
and
\begin{equation}
\hat{\tilde{w}}=\frac{r^{2}}{\tilde{r}^{2}}\hat{w}+\bigg(\frac{r^{2}}{\tilde{r}^{2}}-1\bigg)\sum_{i=3}^{n}(d\phi^{i})^{2}.\label{eq:w_title}
\end{equation}
In particular, we obtain
\begin{equation}
\begin{aligned}\tilde{v}_{n-1} & =v_{n-1}+\frac{(n-2)a+2u_{n-1}}{2(n-1)},\\
\tilde{v}_{n} & =v_{n}+\frac{(n-1)(\tilde{r}_{0}^{n}-r_{0}^{n})+2u_{n}}{2n},\\
\tilde{w}_{n-1} & =w_{n-1}+\frac{2u_{n-1}-a}{2(n-1)}\sum_{i=3}^{n}(d\phi^{i})^{2},\\
\tilde{w}_{n} & =w_{n}+\frac{-(\tilde{r}_{0}^{n}-r_{0}^{n})+2u_{n}}{2n}\sum_{i=3}^{n}(d\phi^{i})^{2}.
\end{aligned}
\label{eq:change_v_and_w}
\end{equation}
As a result, (\ref{eq:scalar_cur_L1_condition})
becomes
\begin{equation}
\tilde{v}_{n-1}+\text{tr}_{h_{0}}\tilde{w}_{n-1}=0,\label{eq:R+n(n-1)_integrable_new_coordinates}
\end{equation}
and (\ref{eq:mass_diff}) becomes
\begin{equation}
\begin{aligned}
E(g)-E(g_{\text{HM}})=\big(\breve{r}_{0}^{n}-\tilde{r}_{0}^{n}\big)\text{Vol}(\partial_{\infty}M,h_{0})
                      +2n\int_{\partial_{\infty}M}\big(\tilde{v}_{n}+\text{tr}_{h_{0}}\tilde{w}_{n}\big)dV_{h_{0}}.\label{eq:mass_diff_2}
\end{aligned}
\end{equation}

\begin{lem}\label{lem:v_tilde_hat_r0}
Under the above condition, it holds that
\begin{equation}
e^{\hat{\tilde{v}}(\tilde{r}_{0},\xi)}=\frac{\breve{r}_{0}}{\tilde{r}_{0}}   \label{eq:v_tilde_hat_r0}
\end{equation}
for all $\xi$.
\end{lem}

\begin{proof}
By (\ref{eq:v_title}) and (\ref{eq:regularity_cond}), we have
\[
e^{\hat{\tilde{v}}(\tilde{r}_{0},\xi)}=\frac{e^{2\hat{u}(r_{+})}}{\frac{d\tilde{r}}{dr}(r_{+})}.
\]
Since (\ref{eq:change_coordinates_grr}) gives that
\[
\bigg(\frac{d\tilde{r}}{dr}\bigg)^{2}=\frac{\tilde{r}^{2}\bigg(1-\frac{\tilde{r}_{0}^{n}}{\tilde{r}^{n}}\bigg)}{r^{2}\bigg(1+\frac{a}{r^{n-1}}-\frac{r_{0}^{n}}{r^{n}}\bigg)}e^{2\hat{u}(r)},
\]
it follows
\begin{align*}
\bigg(\frac{d\tilde{r}}{dr}\bigg)^{2}(r_{+}) & =\frac{\tilde{r}_{0}^{2}}{r_{+}^{2}}e^{2\hat{u}(r_{+})}\lim_{r\rightarrow r_{+}}\frac{1-\frac{\tilde{r}_{0}^{n}}{\tilde{r}^{n}}}{1+\frac{a}{r^{n-1}}-\frac{r_{0}^{n}}{r^{n}}}=\frac{n\tilde{r}_{0}e^{2\hat{u}(r_{+})}\frac{d\tilde{r}}{dr}(r_{+})}{-\frac{(n-1)a}{r_{+}^{n-2}}+\frac{r_{0}^{n}}{r_{+}^{n-1}}}.
\end{align*}
Recalling that $r_{+}$ is a root of $1+\frac{a}{r^{n-1}}-\frac{r_{0}^{n}}{r^{n}}=0$,
this implies
\[
\frac{d\tilde{r}}{dr}(r_{+})=e^{2\hat{u}(r_{+})}\frac{n\tilde{r}_{0}}{r_{+}\big(n-1+\frac{1}{r_{+}^{n}}\big)}.
\]
By (\ref{eq:beta0=beta}), we obtain
\[
\frac{d\tilde{r}}{dr}(r_{+})=e^{2\hat{u}(r_{+})}\frac{\tilde{r}_{0}}{\breve{r}_{0}}.
\]
So the lemma follows.
\end{proof}

Now we are ready to prove Theorem \ref{thm:Main}.

\begin{proof}[Proof of Theorem \ref{thm:Main}]

Denote $\tilde{W}^{r}$, $\hat{\tilde{W}}^{r}$, $\tilde{\mathring{W}}^{r}$ for (\ref{eq:g_rewrite_new_coordinates})
the analogous quantities $W^{r}$, $\hat{W}^{r}$, $\mathring{W}^{r}$ for (\ref{eq:g}). In the new
coordinate function $\tilde{r}$, we have
\begin{align*}
R(g)={} & -n(n-1)+R(\gamma)\\
      & -2e^{-\hat{\tilde{v}}-(n-1)\ln\tilde{r}}\partial_{\tilde{r}}\Big[e^{-2\tilde{u}+\hat{\tilde{v}}}\tilde{r}^{n-1}\big(\partial_{\tilde{r}}\hat{\tilde{v}}+\hat{\tilde{W}}^{r}\big)\Big]\\
      & -\frac{n\tilde{r}_{0}^{n}}{\tilde{r}^{n-1}}\partial_{\tilde{r}}\hat{\tilde{v}}-2e^{2\tilde{u}-\hat{\tilde{v}}}\partial_{\xi}\big(e^{-\hat{\tilde{v}}}\tilde{W}^{\xi}\big)\\
      & -2e^{-2\tilde{u}}\Big[\frac{n-1}{2(n-2)}\big(\hat{\tilde{W}}^{r}\big)^{2}+\frac{1}{8}\vert\tilde{\mathring{W}}^{r}\vert_{\gamma}^{2}\Big]\\
      & -2e^{-2\tilde{v}}\Big[\frac{n-1}{2(n-2)}(\tilde{W}^{\xi})^{2}+\frac{1}{8}\vert\tilde{\mathring{W}}^{\xi}\vert_{\gamma}^{2}\Big].
\end{align*}
This equality can be rewritten as
\beq
\begin{aligned}
-2\partial_{\tilde{r}}\bigg[e^{-2\tilde{u}+\hat{\tilde{v}}}\tilde{r}^{n-1}\big(\partial_{\tilde{r}}\hat{\tilde{v}}+\hat{\tilde{W}}^{r}\big)\bigg]
={}& \;n\tilde{r}_{0}^{n}\partial_{\tilde{r}}\big(e^{\hat{\tilde{v}}}\big)+2e^{2\tilde{u}}\tilde{r}^{n-1}\partial_{\xi}\big(e^{-\hat{\tilde{v}}}\tilde{W}^{\xi}\big)\\
 & -e^{\hat{\tilde{v}}}\tilde{r}^{n-1}R(\gamma)+A\label{eq:scalar_curvature_good_form-2-1}
\end{aligned}
\eeq
where
\begin{align*}
A ={} & e^{\hat{\tilde{v}}}\tilde{r}^{n-1}\big[R(g)+n(n-1)\big]\\
    & +2e^{-2\tilde{u}+\tilde{v}}\tilde{r}^{n-1}\bigg[\frac{n-1}{2(n-2)}\big(\hat{\tilde{W}}^{r}\big)^{2}+\frac{1}{8}\vert\tilde{\mathring{W}}^{r}\vert_{\gamma}^{2}\bigg]\\
    & +2e^{2\tilde{u}-\hat{\tilde{v}}}\tilde{r}^{n-1}\bigg[\frac{n-1}{2(n-2)}(\tilde{W}^{\xi})^{2}+\frac{1}{8}\vert\tilde{\mathring{W}}^{\xi}\vert_{\gamma}^{2}\bigg]\geq0.
\end{align*}
Integrating (\ref{eq:scalar_curvature_good_form-2-1}) on $[\tilde{r}_{0},\infty)$, noting that $e^{-2\tilde{u}}$ vanishes at $\tilde{r}=\tilde{r}_{0}$, and using (\ref{eq:v_tilde_hat_r0}), we obtain
\beq
\begin{aligned}
& \lim_{\tilde{r}\to\infty}-2e^{-2\tilde{u} +\hat{\tilde{v}}}\tilde{r}^{n-1}\big(\partial_{\tilde{r}}\hat{\tilde{v}}+\hat{\tilde{W}}^{r}\big) =n\tilde{r}_{0}^{n}\bigg(1-\frac{\breve{r}_{0}}{\tilde{r}_{0}}\bigg)\\
& +\int_{\tilde{r}_{0}}^{\infty}\bigg[2e^{2\tilde{u}}\tilde{r}^{n-1}\partial_{\xi}\big(e^{-\hat{\tilde{v}}}\tilde{W}^{\xi}\big)-e^{\hat{\tilde{v}}}\tilde{r}^{n-1}R(\gamma)+A\bigg]d\tilde{r}.\label{eq:integration_scalar_curv}
\end{aligned}
\eeq
Since
\begin{align*}
\hat{\tilde{W}}^{r} & =\frac{1}{2}\text{tr}_{\gamma}(\partial_{\tilde{r}}\gamma)-\frac{n-2}{\tilde{r}}\\
 & =-\frac{n-1}{\tilde{r}^{n}}\text{tr}_{h_{0}}\tilde{w}_{n-1}-\frac{n}{\tilde{r}^{n+1}}\text{tr}_{h_{0}}\tilde{w}_{n}+O(\tilde{r}^{-2-n}).
\end{align*}
Combining this with (\ref{eq:v_new_coordinates}) and (\ref{eq:R+n(n-1)_integrable_new_coordinates}),
we obtain
\begin{align*}
\partial_{\tilde{r}}\hat{\tilde{v}}+\hat{\tilde{W}}^{r} & =-\frac{n-1}{\tilde{r}^{n}}\big(\tilde{v}_{n-1}+\text{tr}_{h_{0}}\tilde{w}_{n-1}\big)-\frac{n}{\tilde{r}^{n+1}}\big(\tilde{v}_{n}+\text{tr}_{h_{0}}\tilde{w}_{n}\big)+O(\tilde{r}^{-2-n})\\
 & =-\frac{n}{\tilde{r}^{n+1}}\big(\tilde{v}_{n}+\text{tr}_{h_{0}}\tilde{w}_{n}\big)+O(\tilde{r}^{-2-n}),
\end{align*}
which implies
\[
\lim_{\tilde{r}\to\infty}-2e^{-2\tilde{u}+\hat{\tilde{v}}}\tilde{r}^{n-1}\big(\partial_{\tilde{r}}\hat{\tilde{v}}+\hat{\tilde{W}}^{r}\big)=2n\big(\tilde{v}_{n}+\text{tr}_{h_{0}}\tilde{w}_{n}\big).
\]
Put this into (\ref{eq:integration_scalar_curv}), and then integrate
over $(\partial_{\infty}M,h_{0})$. Noting that the term containing
$\partial_{\xi}\big(e^{-\hat{\tilde{v}}}\tilde{W}^{\xi}\big)$ cancels
away after integrating over $\xi$, and recalling the assumption (\ref{eq:Main_thm_scalar_cur_torus}),
we obtain
\beQ
\begin{aligned}
2n\int_{\partial_{\infty}M}\big(\tilde{v}_{n}+\text{tr}_{h_{0}}\tilde{w}_{n}\big)dV_{h_{0}}\geq{}& n\tilde{r}_{0}^{n}\bigg(1-\frac{\breve{r}_{0}}{\tilde{r}_{0}}\bigg)\text{Vol}(\partial_{\infty}M,h_{0})\\
   &+\int_{\partial_{\infty}M}\int_{\tilde{r}_{0}}^{\infty}Ad\tilde{r}dV_{h_{0}}.
\end{aligned}
\eeQ
Inserting this into (\ref{eq:mass_diff_2}), we obtain
\beQ
\begin{aligned}
E(g)-E(g_{\text{HM}})\geq{}&\tilde{r}_{0}^{n}\bigg[n-1+\bigg(\frac{\breve{r}_{0}}{\tilde{r}_{0}}\bigg)^{n}-n\frac{\breve{r}_{0}}{\tilde{r}_{0}}\bigg]\text{Vol}(\partial_{\infty}M,h_{0})\\
                      &+\int_{\partial_{\infty}M}\int_{\tilde{r}_{0}}^{\infty}Ad\tilde{r}dV_{h_{0}}.
\end{aligned}
\eeQ
It is easy to see that
\[
n-1+s^{n}-ns=(1-s)(n-1-s-\cdots-s^{n-1})\geq0
\]
for $s>0$, and the equality occurs if and only if $s=1$. Taking $s=\frac{\breve{r}_{0}}{\tilde{r}_{0}}$, we obtain
\[
E(g)-E(g_{\text{HM}})\geq0.
\]

For the case of equality, we have
\beQ
\tilde{r}_{0}=\breve{r}_{0}, \quad A\equiv0,
\eeQ
which implies
\[
R+n(n-1)=0,\quad\hat{\tilde{W}}^{r}=\tilde{W}^{\xi}=0,\quad\tilde{\mathring{W}}^{r}=0,\quad\tilde{\mathring{W}}^{\xi}=0.
\]
It follows that
\beQ
\tilde{W}^{r}=\frac{n-2}{\tilde{r}}+\hat{\tilde{W}}^{r}=\frac{n-2}{\tilde{r}}.
\eeQ
Hence
\[
\partial_{r}(\tilde{\gamma}_{ij})=\frac{2\tilde{W}^{r}}{n-2}\tilde{\gamma}_{ij}+\tilde{\mathring{W}}^{r}=\frac{2}{\tilde{r}}\tilde{\gamma}_{ij}.
\]
Therefore
\beQ
\partial_{r}\left(\frac{1}{\tilde{r}^{2}}\tilde{\gamma}_{ij}\right)=0.
\eeQ
It implies that
\beQ
\frac{1}{\tilde{r}^{2}}\tilde{\gamma}_{ij}=\frac{1}{\tilde{r}^{2}}\tilde{\gamma}_{ij}\Big\vert _{\tilde{r}=\infty}=\delta_{ij}, \quad \gamma=\tilde{r}^{2}\sum_{i=3}^{n}(d\phi^{i})^{2}.
\eeQ
Now we can simplify (\ref{eq:scalar_curvature_good_form-2-1}) and find that
\[
\partial_{\tilde{r}}^{2}\bigg(e^{\hat{\tilde{v}}}\bigg)+\frac{(n+1)\tilde{r}^{n}+(\frac{n}{2}-1)\tilde{r}_{0}^{n}}{\tilde{r}\big(\tilde{r}^{n}-\tilde{r}_{0}^{n}\big)}\cdot\bigg(\partial_{\tilde{r}}e^{\hat{\tilde{v}}}\bigg)^{2}=0.
\]
Since
\beQ
e^{\hat{\tilde{v}}(\tilde{r}_{0},\xi)}=\frac{\breve{r}_{0}}{\tilde{r}_{0}}=1=\lim_{\tilde{r}\to\infty}e^{\hat{\tilde{v}}(\tilde{r})},
\eeQ
we can show $\hat{\tilde{v}}\equiv0$ by the maximum principle. Finally, we obtain
\[
g=\frac{1}{\tilde{r}^{2}\bigg(1-\frac{\breve{r}_{0}^{n}}{\tilde{r}^{n}}\bigg)}d\tilde{r}^{2}+\tilde{r}^{2}\bigg(1-\frac{1}{\tilde{r}^{n}}\bigg)d\xi^{2}+\tilde{r}^{2}\sum_{i=3}^{n}(d\phi^{i})^{2}.
\]
Thus $g$ is the Horowitz--Myers metric $g_{\text{HM}}$.
\end{proof}

\bigskip

{\footnotesize {\it Acknowledgement. The work is supported by Chinese National Natural Science of Foundation (Grant No. 11731001), the special foundation for Junwu and Guangxi Ba Gui Scholars, the Guangdong Basic and Applied Basic Research Foundation 2021A1515010379. The authors thank the referees for carefully reading the manuscript and for their helpful suggestions.}

\end{document}